\documentclass[11pt]{amsart}
\usepackage{amssymb,amsmath,amsthm,bbm}

\usepackage{graphics}

\makeatletter
  \def\title@font{\Large\bfseries}
  \let\ltx@maketitle\@maketitle
  \def\@maketitle{\bgroup%
    \let\ltx@title\@title%
    \def\@title{\resizebox{\textwidth}{!}{%
      \mbox{\title@font\ltx@title}%
    }}%
    \ltx@maketitle%
  \egroup}
\makeatother

\usepackage{color}

\newtheorem{theorem}{Theorem}[section]
\newtheorem{lemma}[theorem]{Lemma}
\newtheorem{prop}[theorem]{Proposition}

\theoremstyle{definition}
\newtheorem{definition}[theorem]{Definition}

\newtheorem{remark}[theorem]{Remark}

 \theoremstyle{plain}
\newtheorem*{namedthm}{\namedthmname}
\newcounter{namedthm}

\makeatletter
\newenvironment{named}[1]
  {\def\namedthmname{#1}%
   \refstepcounter{namedthm}%
   \namedthm\def\@currentlabel{#1}}
  {\endnamedthm}
  \makeatother

\newcommand{\RR}{\mathbb{R}}

\newcommand{\Ec}{\mathcal{E}}
\newcommand{\Htheta}{\mathcal{H}_{\theta}}

\newcommand{\ddbar}{\partial \bar{\partial}}

\newcommand{\PSH}{{\rm PSH}}
\newcommand{\setdef}{\; | \; }
\newcommand{\vol}{{\rm Vol}}
\newcommand{\Amp}{{\rm Amp}}
\newcommand{\id}{\mathbbm{1}}

 \numberwithin{equation}{section}

\usepackage{hyperref}
\hypersetup{
    bookmarks=true,         
    unicode=false,         
    pdftoolbar=true,       
    pdfmenubar=true,       
    pdffitwindow=false,    
    pdfstartview={FitH},   
    pdftitle={Lp metric geometry of big and nef cohomology classes},    
    pdfauthor={E. Di Nezza, C.H. Lu},     
    colorlinks=true,       
   linkcolor=black,         
    citecolor=black,        
    filecolor=black,      
    urlcolor=black}          

\title{$L^p$ metric geometry of big and nef cohomology classes}

\author{Eleonora Di Nezza}
\address{Eleonora Di Nezza, Institut des Hautes \'Etudes Scientifiques
Universit\'e Paris-Saclay}
\email{dinezza@ihes.fr}
\urladdr{\href{https://sites.google.com/site/edinezza/home}{https://sites.google.com/site/edinezza/home}}

\author{Chinh H. Lu}
\address{Hoang-Chinh Lu, Laboratoire de Math\'ematiques d'Orsay,
 Univ. Paris-Sud,
 CNRS, Universit\'e Paris-Saclay,
  91405 Orsay, France}
\email{hoang-chinh.lu@math.u-psud.fr}
\urladdr{\href{https://www.math.u-psud.fr/~lu/}{https://www.math.u-psud.fr/~lu/}}

\date{\today}
\thanks{The authors are partially supported by the French ANR project GRACK}

\begin{document}

\maketitle

\begin{center}
{\it In honor of L\^e V\u{a}n Thi\^em's centenary} 
\end{center}

\begin{abstract}
Let $(X,\omega)$ be a compact K\"ahler manifold of dimension $n$, and $\theta$ be a closed smooth real $(1,1)$-form representing a big and nef cohomology class. We introduce a metric $d_p, p\geq 1$, on the finite energy space $\mathcal{E}^p(X,\theta)$, making it a complete geodesic metric space. 	
\end{abstract}

\tableofcontents

\section{Introduction}
Finding canonical  (K\"ahler-Einstein, cscK, extremal) metrics  on compact K\"ahler manifolds  is one of the central questions in differential geometry (see \cite{Calabi_1982_Seminar}, \cite{Yau_1978_CPAM}, \cite{Szekelyhidi_2014_book} and the references therein).  Given a K\"ahler metric $\omega$ on a compact K\"ahler manifold $X$, one looks for a K\"ahler potential $\varphi$ such that $\omega_{\varphi}:= \omega+dd^c\varphi$  is ``canonical''. Mabuchi introduced a Riemannian structure on the space of K\"ahler potentials $\mathcal{H}_{\omega}$.  As shown by Chen \cite{Chen_2000_JDG} $\mathcal{H}_{\omega}$ endowed with the Mabuchi $d_2$ distance is a metric space.  Darvas \cite{Darvas_2017_AJM} showed that its metric completion coincides with a finite energy class of plurisubharmonic functions introduced by Guedj and Zeriahi \cite{Guedj_Zeriahi_2007_JFA}.  Other Finsler geometries $d_p$, $p\geq 1$, on $\mathcal{H}_{\omega}$ were studied by Darvas \cite{Darvas_2015_AIM} and they lead to several spectacular results related to a longstanding conjecture on existence  of cscK metrics and properness of K-energy (see \cite{Darvas_Rubinstein_2017_JAMS}, \cite{Berman_Darvas_Lu_2016_Minimizer},  \cite{Chen_Cheng_2017_csckestimates,Chen_Cheng_2017_csckexistence,Chen_Cheng_2017_csckgeneral}). Employing the same technique as in \cite{Darvas_Rubinstein_2017_JAMS} and extending the $L^1$-Finsler structure of  \cite{Darvas_2015_AIM} to big and semipositive classes via a formula relating the Monge-Amp\`ere energy and the $d_1$ distance, Darvas \cite{Darvas_2017_IMRN} established analogous results for singular normal K\"ahler varieties. Motivated by the same geometric applications, the $L^p$ ($p\geq 1$) Finsler geometry in big and semipositive cohomology classes was constructed in \cite{DiNezza_Guedj_2016_CM} via an approximation method.

In this note we extend the main results of \cite{Darvas_2015_AIM,DiNezza_Guedj_2016_CM} to the context of big and nef cohomology classes. Assume that $X$ is a compact K\"ahler manifold of complex dimension $n$ and let $\theta$ be a smooth closed real $(1,1)$ form representing a  big \& nef cohomology class. Fix $p\geq 1$.

\begin{named}{Main Theorem}\label{mainthm}
	The space $\Ec^p(X,\theta)$ endowed with $d_p$ is a complete geodesic metric space.  
\end{named} 

For the definition of $\Ec^p(X,\theta)$, $d_p$ and relevant notions we refer to Section  \ref{sect: preliminaries}. 
When $p=1$  \ref{mainthm} was established in \cite{Darvas_Dinezza_Lu_2018_L1} in the more general case of big cohomology classes using the approach of \cite{Darvas_2017_IMRN}. Here, we use an approximation argument as in \cite{DiNezza_Guedj_2016_CM} with an important modification due to the fact that generally potentials in big  cohomology classes are unbounded. Interestingly, this modification greatly  simplifies the proof of \cite[Theorem A]{DiNezza_Guedj_2016_CM}. 
\medskip

\paragraph{\bf Organization of the note. }  We recall relevant notions in pluripotential theory in big cohomology classes in Section \ref{sect: preliminaries}. The metric space $(\Ec^p,d_p)$ is introduced in Section \ref{section:  distance dp} where we prove \ref{mainthm}. In case $p=1$ we show in Proposition \ref{prop: approximation vs MA} that the distance $d_1$ defined in this note and the one defined in \cite{Darvas_Dinezza_Lu_2018_L1} do coincide.   

\subsection*{Acknowledgements}
We thank Tam\'as Darvas for valuable discussions.

\section{Preliminaries}\label{sect: preliminaries}
Let $(X,\omega)$ be a compact K\"ahler manifold of dimension $n$.  We use the following real differential operators $d= \partial +\bar{\partial}$, $d^c = i(\bar{\partial}-\partial)$, so that $dd^c =2i \ddbar$. We briefly recall known results in pluripotential theory in big cohomology  classes, and refer the reader to \cite{Boucksom_Eyssidieux_Guedj_Zeriahi_2010_AM}, \cite{Berman_Boucksom_Guedj_Zeriahi_2013_IHES}, \cite{Darvas_DiNezza_Lu_2018_CM,Darvas_Dinezza_Lu_2018_APDE,Darvas_Dinezza_Lu_2018_L1,Darvas_Dinezza_Lu_2018_Logconcave} for more details. 
\subsection{Quasi-plurisubharmonic functions}
 A function $u: X \rightarrow \RR \cup \{-\infty\}$ is quasi-plurisubharmonic (or quasi-psh) if it is locally the sum of a psh function and a smooth function. Given a smooth closed real  $(1,1)$-form $\theta$, we let $\PSH(X,\theta)$ denote the set of all integrable quasi-psh functions $u$ such that $\theta_u:= \theta+dd^c u \geq 0$, where the inequality is understood in the sense of currents. A function $u$ is said to have analytic singularities if locally $u=\log \sum_{j=1}^N |f_j|^2 + h$, where the $f_j's$ are holomorphic and $h$ is smooth. 
 
 The De Rham cohomology class $\{\theta\}$ is K\"ahler if it contains a K\"ahler potential, i.e. a function $u\in \PSH(X,\theta)\cap \mathcal{C}^{\infty}(X,\RR)$ such that $\theta+dd^c u>0$. The class $\{\theta\}$ is nef if $\{\theta +\varepsilon \omega\}$ is K\"ahler for all $\varepsilon>0$. It is pseudo-effective if the set $\PSH(X,\theta)$ is non-empty, and big if $\{\theta-\varepsilon \omega\}$ is pseudo-effective for some $\varepsilon>0$.
 The ample locus of $\{\theta\}$, which will be denoted by $\Amp(\theta)$, is the set of all points $x\in X$ such that there exists $\psi \in \PSH(X,\theta-\varepsilon \omega)$ with analytic singularities and smooth in a neighborhood of $x$. It was shown in \cite[Theorem 3.17]{Boucksom_2004_ASENS} that $\{\theta\}$ is K\"ahler iff $\Amp(\theta)=X$. 
 
 Throughout this paper we always assume that $\{\theta\}$ is big and nef.  Typically, there are no bounded functions in  $\PSH(X,\theta)$, but there are plenty of locally bounded functions as we now briefly recall.   By the bigness of $\{\theta\}$ there exists $\psi \in \PSH(X,\theta-\varepsilon \omega)$ for some $\varepsilon>0$. Regularizing $\psi$ (by \cite[Main Theorem 1.1]{Demailly_1992_JAG}) we can find a function $u\in \PSH(X,\theta-\frac{\varepsilon}{2} \omega)$ smooth in a Zariski open set $\Omega$ of $X$.   Roughly speaking, $\theta_u$ locally behaves as a K\"ahler form on $\Omega$. As shown in \cite[Theorem 3.17]{Boucksom_2004_ASENS} $u$ and $\Omega$ can be constructed in such a way that $\Omega$ is the ample locus of $\{\theta\}$.

If $u$ and $v$ are two  $\theta$-psh functions on $X$, then $u$ is said to be \emph{less singular} than $v$ if $v\leq u+C$ for some $C\in \Bbb R$, while they are said to have the \emph{same singularity type} if $u-C \leq v\leq u+C$, for some $C\in \mathbb{R}$. A  $\theta$-psh function $u$ is said to have \emph{minimal singularities} if it is less singular than any other $\theta$-psh function. An example of a $\theta$-psh function with minimal singularities is 
$$V_\theta:=\sup\{ u\in \PSH(X, \theta)  \setdef  u\leq 0\}.$$
For a function $f: X \rightarrow  \mathbb{R}$, we let $f^*$ denote its upper semicontinuous regularization, i.e. 
$$
f^*(x) := \limsup_{X\ni y \to x} f(y).
$$
Given a measurable function $f$ on $X$ we define 
\[
P_{\theta}(f) := \left ( x \mapsto \sup \{u(x) \setdef u \in \PSH(X,\theta), \  u \leq f \}\right)^*.
\]
 We will need the following result of Berman \cite{Berman_2018_MZ}:
\begin{theorem}\label{thm: regularity of psh envelope}
Let $f$ be a continuous function such that $dd^c f\leq C\omega $ on $X$, for some $C>0$. Then $\Delta_{\omega} (P_{\theta}(f))$ is locally bounded on $\Amp(\theta)$, and 
\begin{equation*}
(\theta+dd^c P_{\theta}(f))^n= \mathbf{1}_{\{P_{\theta}(f)=f\}} (\theta+dd^c f)^n.
\end{equation*}
If $\theta$ is moreover K\"ahler then $\Delta_{\omega}(P_{\theta}(f))$ is globally bounded on $X$. 
\end{theorem}

If $f=\min(u,v)$ for $u,v$ quasi-psh then there is no need to take the upper semicontinuous regularization in the definition of $P(u,v):= P_{\theta}(\min(u,v))$. The latter is the largest $\theta$-psh function lying below both $u$ and $v$, and is called the rooftop envelope of $u$ and $v$ in \cite{Darvas_Rubinstein_2016_JMSJ}.

\subsection{Non-pluripolar Monge-Amp\`ere products}

Given  $u_1,...,u_p$ $\theta$-psh functions with minimal singularities, $\theta_{u_1}\wedge ... \wedge \theta_{u_p}$, as defined by Bedford and Taylor \cite{Bedford_Taylor_1976_IM,Bedford_Taylor_1982_AM} is  a closed  positive current in $\Amp(\theta)$.   For general $u_1,...,u_p \in \PSH(X,\theta)$, it was shown in \cite{Boucksom_Eyssidieux_Guedj_Zeriahi_2010_AM} that the \emph{non-pluripolar product} of $\theta_{u_1},\ldots,\theta_{u_p}$, that we still denote by
$$
\theta_{u_1}\wedge \ldots \wedge \theta_{u_p},
$$
is well-defined as  a  closed positive $(p,p)$-current on $X$ which does not charge pluripolar sets. For a $\theta$-psh function $u$, the \emph{non-pluripolar complex Monge-Amp{\`e}re measure} of $u$ is simply $\theta_u^n:=\theta_u\wedge \ldots\wedge \theta_u$.  

If $u$ has minimal singularities then $\int_X \theta_u^n$, the total mass of $\theta_u^n$, is equal to $\int_X \theta_{V_{\theta}}^n$,  the volume of the class $\{\theta\}$ denoted by $\vol(\theta)$. For a general $u\in \PSH(X,\theta)$,  $\int_X \theta_u^n$ may take any value in $[0,\vol(\theta)]$. Note that $\vol(\theta)$ is a cohomological quantity, i.e. it does not depend on the smooth representative we choose in $\{\theta\}$. 
\subsection{The energy classes} From now on, we fix $p\geq 1$.

Recall that for any $\theta$-psh function $u$ we have $\int_X \theta_u^n \leq \vol(\theta)$. 
We denote by $\mathcal{E}(X, \theta)$ the set of $\theta$-psh functions $u$ such that $\int_X \theta_u^n = \vol(\theta)$. We let $\Ec^p(X,\theta)$ denote the set of $u\in \Ec(X,\theta)$ such that $\int_X |u-V_{\theta}|^p \theta_u^n<+\infty$. 
For $u,v \in \mathcal{E}^p(X, \theta)$ we define
$$
I_p(u,v):=I_{p,\theta}(u,v):=\int_X |u-v|^p \left( \theta_{u}^n+\theta_{v}^n\right).
$$
It was proved in \cite[Theorem 1.6]{Guedj_Lu_Zeriahi_2017_JDG} that $I_p$ satisfies a quasi triangle inequality:
$$
I_{p,\theta}(u,v) \leq C(n,p) (I_{p,\theta}(u,w)+ I_{p,\theta}(v,w)), \ \forall u,v,w\in \Ec^p(X,\theta). 
$$
In particular, applying this for $w=V_{\theta}$ and using Theorem \ref{thm: regularity of psh envelope} we obtain $I_{p,\theta}(u,v)<+\infty$, for all $u,v\in \Ec^p(X,\theta)$. 
Moreover, it follows from the domination principle  \cite[Proposition 2.4]{Darvas_DiNezza_Lu_2018_CM} that $I_p$ is non-degenerate:
$$
I_{p,\theta}(u,v)=0 \Longrightarrow u=v.
$$

\subsection{Weak geodesics}\label{sec geo}
Geodesic segments connecting K\"ahler potentials  were first introduced by Mabuchi \cite{Mabuchi_1987_OJM}. Semmes \cite{Semmes_1992_AJM} and Donaldson \cite{Donaldson_1999_bookAMS} independently realized that the geodesic equation can be reformulated as a degenerate homogeneous complex Monge-Amp\`ere equation.  The best regularity of a geodesic segment connecting two K\"ahler potentials is known to be $\mathcal{C}^{1,1}$ (see \cite{Chen_2000_JDG}, \cite{Blocki_2012_bookIP}, \cite{Chu_Tosatti_Weinkove_2017_AoPDE}).

In the context of a big cohomology class, the regularity of geodesics is very delicate. To avoid this issue  we follow an idea of  Berndtsson \cite{Berndtsson_2015_IM} considering geodesics as the upper envelope of subgeodesics  (see \cite{Darvas_DiNezza_Lu_2018_CM}). 

For a curve $[0,1] \ni t \mapsto u_t \in \PSH(X,\theta)$  we define 
$$
X\times D \ni (x,z) \mapsto U(x,z) := u_{\log |z|}(x),	
$$
where $D:= \{z\in \mathbb{C} \setdef 1< |z|<e \}$. We let $\pi: X\times D \rightarrow X$ be the projection on $X$. 
\begin{definition}
	We say that $t\mapsto u_t$ is a subgeodesic if $(x,z) \mapsto U(x,z)$ is a $\pi^{*}\theta$-psh function on $X\times D$.
\end{definition}

\begin{definition}
	For  $\varphi_0,\varphi_1 \in \PSH(X,\theta)$, we let $\mathcal{S}_{[0,1]}(\varphi_0,\varphi_1)$ denote the set of  all subgeodesics $[0,1] \ni t \mapsto u_t$ such that $\limsup_{t\to 0} u_t\leq \varphi_0$ and $\limsup_{t\to 1} u_t\leq \varphi_1$. 
\end{definition}

Let $\varphi_0,\varphi_1 \in \PSH(X,\theta)$.  We define, for $(x,z)\in X\times D$, 
$$
\Phi(x,z) :=  \sup \{ U(x,z) \setdef U \in \mathcal{S}_{[0,1]}(\varphi_0,\varphi_1) \}. 
$$
The curve $t\mapsto \varphi_t$ constructed from $\Phi$ via \eqref{eq: complexified curve} is called the weak Mabuchi geodesic connecting $\varphi_0$ and $\varphi_1$.  

Geodesic segments connecting two general $\theta$-psh functions may not exist. If $\varphi_0, \varphi_1 \in \mathcal{E}^p(X,\theta)$, it was shown in \cite[Theorem 2.13]{Darvas_DiNezza_Lu_2018_CM} that $P(\varphi_0,\varphi_1) \in \mathcal{E}^p(X,\theta)$. Since $P(\varphi_0,\varphi_1) \leq \varphi_t$, we obtain that $t \to \varphi_t$ is a curve in $\mathcal{E}^p(X,\theta)$. Each subgeodesic segment is in particular convex in $t$:  
\[
\varphi_t\leq \left (1-t\right )\varphi_0 + t\varphi_1, \ \forall t\in [0,1]. 
\]
Consequently the upper semicontinuous regularization (with respect to both variables $x,z$) of $\Phi$ is again in $\mathcal{S}_{[0,1]}(\varphi_0,\varphi_1)$, hence so is $\Phi$.  In particular, if $\varphi_0,\varphi_1$ have minimal singula\-ri\-ties then the geodesic $\varphi_t$ is Lipschitz on $[0,1]$ (see \cite[Lemma 3.1]{Darvas_DiNezza_Lu_2018_CM}): 
$$
|\varphi_t-\varphi_s| \leq |t-s| \sup_{X} |\varphi_0-\varphi_1|, \ \forall t,s \in [0,1]. 
$$

\subsection{Finsler geometry in the K\"ahler case}
 Darvas \cite{Darvas_2015_AIM} introduced a family of distances in the space of K\"ahler potentials 
$$\mathcal{H}_\omega:= \{\varphi\in \mathcal{C}^\infty (X, \mathbb{R}) \setdef \omega_\varphi>0\}.
$$
\begin{definition}
Let $\varphi_0,\varphi_1\in \mathcal{H}_\omega$. For $p \geq 1$, we set
$$
d_p(\varphi_0,\varphi_1):=\inf \{ \ell_p(\psi) \, | \, \psi
\text{ is a smooth path  joining } \varphi_0 \text{ to } \varphi_1 \},
$$
where 
$
\ell_p(\psi):=\int_0^1 \left( \frac{1}{V} \int_X | \dot{\psi}_t|^p \omega_{\psi_t}^n\right)^{1/p} dt
$
and $V:= \vol(\omega)=\int_X \omega^n$.
\end{definition}
It was then proved in \cite[Theorem 1]{Darvas_2015_AIM} (generalizing Chen's original arguments \cite{Chen_2000_JDG}) that $d_p$ defines a distance on $\mathcal{H}_{\omega}$, and for all $\varphi_0, \varphi_1\in \mathcal{H}_\omega$,
\begin{equation}\label{eq: dp formula}
d_p(\varphi_0,\varphi_1)= \left(\frac{1}{V}\int_ X |\dot{\varphi_t}|^p \omega_{\varphi_t}^n \right)^{1/p},  \quad \forall t\in [0,1],
\end{equation}
where $t\rightarrow \varphi_t$ is the Mabuchi geodesic (defined in Section \ref{sec geo}). It was shown in \cite[Lemma 4.11]{Darvas_2015_AIM} that \eqref{eq: dp formula} still holds for $\varphi_0, \varphi_1\in \PSH(X,\omega)$ with $dd^c \varphi_i\leq C\omega, i=0,1$, for some positive constant $C$.

By \cite{Demailly_1992_JAG,Blocki_Kolodziej_2007_PAMS}, potentials in $\Ec^p(X,\omega)$ can be approximated from above by smooth K\"ahler potentials. As shown in \cite{Darvas_2017_AJM} the metric $d_p$ can be extended for potentials in $\varphi_0, \varphi_1\in\mathcal{E}^p(X, \omega)$: if $\varphi_i^k $  are smooth strictly $\omega$-psh functions decreasing to $\varphi_i$, $i=0,1$ then  the limit 
$$d_p(\varphi_0, \varphi_1):=  \lim_{k\rightarrow +\infty} d_p(\varphi_0^k, \varphi_1^k)$$
exists and it is independent of the approximants. By \cite[Lemma 4.4 and 4.5]{Darvas_2015_AIM}, $d_p$ defines a metric on $\mathcal{E}^p(X,\omega)$ and  $(\Ec^p(X,\omega),d_p)$ is a complete geodesic metric space. 

\section{The metric space $(\mathcal{E}^p(X,\theta),d_p)$}\label{section:  distance dp}
The goal of this section is to define a distance $d_p$ on $\mathcal{E}^p(X, \theta)$ and prove that the space $(\Ec^p(X,\theta), d_p)$ is a complete geodesic metric space.  We follow the strategy in \cite{DiNezza_Guedj_2016_CM}, approximating the space of ``K\"ahler potentials'' $\mathcal{H}_{\theta}$ by regular spaces $\mathcal{H}_{\omega_{\varepsilon}}$, where  $\omega_{\varepsilon}:= \theta +\varepsilon \omega$ represents K\"ahler cohomology classes for any $\varepsilon>0$ (by nefness of $\theta$). 
Note that $\omega_\varepsilon$ is not necessarily a K\"ahler form but there exists a smooth potential $f_\varepsilon\in \mathcal{C}^\infty(X,\ \mathbb{R})$ such that $\omega_\varepsilon+dd^c f_\varepsilon$ is a K\"ahler form.  For notational convenience we normalize $\theta$ so that $\vol(\theta)= \int_X \theta_{V_{\theta}}^n=1$ and we set $V_{\varepsilon}:= \vol(\omega_{\varepsilon})$. 

Typically there is no smooth potentials in $\PSH(X,\theta)$ but the following class contains plenty of potentials sufficiently regular for our purposes:
$$\mathcal{H}_\theta:=\{ \varphi\in \PSH(X, \theta) \setdef  \varphi=P_\theta(f), \; f\in \mathcal{C}(X, \mathbb{R}), \; dd^c f\leq C(f) \omega \}.$$
Here $C(f)$ denotes a positive constant which depends also on $f$. Note that any $u=P_{\theta}(f)\in \mathcal{H}_{\theta}$ has minimal singularities because, for some constant $C>0$, $V_{\theta}-C$ is a candidate defining $P_{\theta}(f)$. The following elementary observation will be useful in the sequel. 
\begin{lemma}\label{lem: min}
If $u,v\in \Htheta$ then $P_{\theta}(u,v)\in \Htheta$. 
\end{lemma}
\begin{proof}
Set $h= \min(f,g)\in \mathcal{C}^0(X,\RR)$, where $f,g\in \mathcal{C}^0(X,\RR)$ are such that $u=P_{\theta}(f)$ and $v=P_{\theta}(g)$ and $dd^c f\leq C\omega$, $dd^c g\leq C\omega$. Then $-h=\max(-f,-g)$ is a $C\omega$-psh function on $X$, hence $dd^c (-h) + C\omega\geq 0$.
\end{proof}

\subsection{Defining a distance $d_p$ on $\mathcal{H}_{\theta}$}
By Darvas \cite{Darvas_2015_AIM}, the Mabuchi distance $d_{p, \omega}$ is well defined on $\mathcal{E}^p(X, \omega)$ when the reference form $\omega$ is a K\"ahler form. With the following observation we show that such a distance behaves well when we change the K\"ahler representative in  $\{\omega\}$.

\begin{prop}\label{prop: dp for kahler classes}
Let $\omega_f:=\omega+dd^c f\in \{\omega\}$ be another K\"ahler form. Then, given $\varphi_0, \varphi_1\in \mathcal{E}^p(X, \omega)$ we have
$$d_{p, \omega} (\varphi_0, \varphi_1)= d_{p,\omega_f}(\varphi_0-f, \varphi_1-f).$$
\end{prop}

\begin{proof}
Let $\varphi_t$ be the Mabuchi geodesic (w.r.t $\omega$) joining $\varphi_0$ and $\varphi_1$ and let $\varphi_t^f$ be the Mabuchi geodesic (w.r.t $\omega_f$) joining $\varphi_0-f$ and $\varphi_1-f$. We claim that $\varphi_t^f= \varphi_t-f$. Indeed,  $\varphi_t-f$ is an $\omega_f$-subgeodesic connecting $\varphi_0-f$ and $\varphi_1-f$. Hence $\varphi_t-f\leq \varphi_t^f$. On the other hand $\varphi_t^f +f$ is a candidate defining $\varphi_t$, thus $\varphi_t^f +f \leq \varphi_t$, proving the claim.

Assume $\varphi_0, \varphi_1$ are K\"ahler potentials. By \eqref{eq: dp formula} we have
\begin{flalign*}
Vd_{p, \omega}^p (\varphi_0, \varphi_1)&= \int_X |\dot{\varphi_0}|^p (\omega+dd^c \varphi_0)^n\\
&= \int_X \left|\lim_{t\rightarrow 0^+}  \frac{(\varphi_t-f) -(\varphi_0-f)}{t} \right|^p \, \left(\omega_f+dd^c (\varphi_0-f)\right)^n\\
&= \int_X |\dot{\varphi^f_0}|^p (\omega_f+dd^c (\varphi_0-f))^n\\
&= Vd^p_{p,\omega_f}(\varphi_0-f, \varphi_1-f).
\end{flalign*}
The identity for potentials in $\Ec^p(X,\omega)$ follows from the fact that the distance $d_{p, \omega}$ between potentials $\varphi_0, \varphi_1\in \mathcal{E}^p(X, \omega)$ is defined as the limit $\lim_j d_{p, \omega}(\varphi_{0,j}, \varphi_{1,j})$, where $\{\varphi_{i,j}\}$ is a sequence of smooth strictly $\omega$-psh functions decreasing to $\varphi_i$, for $i=0,1$.
\end{proof}
Thanks to the above Proposition we can then define the Mabuchi distance w.r.t any smooth $(1,1)$-form $\eta$ in the K\"ahler class $\{\omega\}$:
\begin{equation}\label{distance form}
d_{p,\eta}(\varphi_0, \varphi_1):= d_{p, \eta_f} (\varphi_0-f, \varphi_1-f),\qquad \varphi_0, \varphi_1\in \mathcal{E}^p(X,\eta)
\end{equation}
where $\eta_f = \eta +dd^c f$ is a K\"ahler form. Proposition \ref{prop: dp for kahler classes} reveals that the definition is  independent of the choice of $f$. 

We next extend the Pythagorean formula of \cite{Darvas_2015_AIM,Darvas_2017_AJM} for K\"ahler classes. 
\begin{lemma}\label{lem: Pythagorean Darvas}
If $\{\eta\}$ is K\"ahler and $u,v\in \Ec^p(X,\eta)$ then 
$$
d_{p,\eta}^p(u,v) = d_{p,\eta}^p(u,P_{\eta}(u,v)) + d_{p,\eta}^p(v,P_{\eta}(u,v)). 
$$
\end{lemma}
\begin{proof}
	By \cite[Corollary 4.14]{Darvas_2015_AIM} and \eqref{distance form}, we have
$$d_{p,\eta}^p(u,v)= d_{p,\eta_f}^p(u-f,P_{\eta_f}(u-f,v-f)) + d_{p,\eta_f}^p(v-f,P_{\eta_f}(u-f,v-f)).$$
The conclusion follows observing that $P_{\eta_f}(u-f,v-f)= P_\eta (u,v)-f$.
\end{proof}

The following results play a crucial role in the sequel.

\begin{lemma}\label{lem: Ip convergence}
Let  $\varphi=P_\theta(f),\psi=P_\theta(g)  \in \mathcal{H}_\theta $. Set $\varphi_\varepsilon:= P_{\omega_\varepsilon}(f)$ and $\psi_\varepsilon= P_{\omega_\varepsilon}(g)$. Then
$$
\lim_{\varepsilon\rightarrow 0} I_{p, \omega_{\varepsilon} }(\varphi_\varepsilon, 	\psi_\varepsilon)= I_{p,\theta}(\varphi, \psi).
$$
\end{lemma}

\begin{proof}
Observe that $|\varphi_{\varepsilon}-\psi_{\varepsilon}| \to |\varphi-\psi|$ pointwise on $X$ and they are uniformly bounded: 
$$
|\varphi_{\varepsilon}-\psi_{\varepsilon}| \leq \sup_X|f-g|. 
$$
By  Lemma \ref{lem: density converge} below and Lebesgue's dominated convergence theorem, 
\begin{flalign*}
	\lim_{\varepsilon \to 0} \int_X \left| \varphi_\varepsilon- \psi_\varepsilon\right|^p (\omega_\varepsilon+dd^c\varphi_\varepsilon)^n  = \int_X \left| \varphi- \psi\right|^p (\theta+dd^c \varphi)^n.
\end{flalign*}
Similarly, the other term in the definition of $I_{p, \omega_\varepsilon}$ also converges to the desired limit.
\end{proof}

\begin{lemma} \label{lem: density converge}
	Let $\varphi = P_{\theta}(f)\in \mathcal{H}_{\theta}$. For $\varepsilon>0$ we set $\varphi_{\varepsilon}= P_{\omega_{\varepsilon}}(f)$ and write 
	$$
	(\omega_{\varepsilon}+dd^c \varphi_{\varepsilon})^n = \rho_{\varepsilon} \omega^n\ ; \ (\theta+dd^c\varphi)^n = \rho \omega^n.
	$$ 
	Then $\varepsilon\mapsto \rho_{\varepsilon}$ is increasing, uniformly bounded and $\rho_{\varepsilon} \to \rho$ pointwise on $X$.  
\end{lemma}
\begin{proof}
Define, for $\varepsilon>0$, $D_{\varepsilon}:= \{x \in X \setdef \varphi_{\varepsilon}(x)= f(x)\}$. Since $\{\varphi_{\varepsilon}\}$ is increasing and $\varphi_{\varepsilon}\leq f$, $\{D_{\varepsilon}\}$ is also increasing. We set $D:= \cap_{\varepsilon>0} D_{\varepsilon}$. Then $D=\{x \in X \setdef \varphi(x)=f(x)\}$.

	For $\varepsilon'>\varepsilon>0$, it follows from Theorem \ref{thm: regularity of psh envelope} that 
\begin{flalign*}
(\omega_{\varepsilon}+dd^c \varphi_{\varepsilon})^n & = \id_{\{\varphi_{\varepsilon}=f\}} (\omega_{\varepsilon}+dd^c f)^n\\
& \leq  \id_{\{\varphi_{\varepsilon}=f\}}(\omega_{\varepsilon'}+dd^c f)^n \leq (\omega_{\varepsilon'}+dd^c \varphi_{\varepsilon'})^n. 	
\end{flalign*}
Here we use the fact that $0 \leq \omega_{\varepsilon} +dd^c f \leq \omega_{\varepsilon'} +dd^c f$ on $D_{\varepsilon}$. This proves the first statement. The second statement follows from the bound $dd^c f\leq C\omega$. We now prove the last statement. If $x\in D$, using $(\theta+dd^c f)\leq C' \omega$ we can write 
\begin{flalign*}
\rho_{\varepsilon}(x) \omega^n &= (\theta + \varepsilon \omega +dd^c f)^n \leq (\theta +dd^c f)^n + O(\varepsilon) \omega^n\\
&= (\rho(x)+O(\varepsilon))\omega^n. 	
\end{flalign*}
Hence $\rho_{\varepsilon}(x) \to \rho(x)$. If $x\notin D$ then $x\notin D_{\varepsilon}$ for $\varepsilon>0$ small enough, hence $\rho_{\varepsilon}(x)=0=\rho(x)$. 
\end{proof}

\begin{lemma}
	\label{lem: time derivative converge}
	Let $\varphi_j= P_{\theta}(f_j) \in \mathcal{H}_{\theta}$, for $j=0,1$. Let $\varphi_t$ (resp. $\varphi_{t,\varepsilon}$) be weak Mabuchi geodesics joining $\varphi_0$ and $\varphi_1$ (resp. $\varphi_{0,\varepsilon}= P_{\omega_{\varepsilon}}(f_0)$ and $\varphi_{1,\varepsilon}= P_{\omega_{\varepsilon}}(f_1)$). Then we have the following pointwise convergence
	$$
	\id_{\{\varphi_{0,\varepsilon} =f_0\}} |\dot{\varphi}_{0,\varepsilon}|^p \rightarrow \id_{\{\varphi_{0} =f_0\}} |\dot{\varphi}_{0}|^p.
	$$
\end{lemma}
\begin{proof}
	Since $P_{\omega_{\varepsilon}}(f_j)\geq P_{\theta}(f_j)$, $j=0,1$, it follows from the definition that $\varphi_{t,\varepsilon} \geq \varphi_t$ (the curve $\varphi_t$ is a candidate defining $\varphi_{t,\varepsilon}$ for any $\varepsilon>0$). Set $D_{\varepsilon}=\{\varphi_{0,\varepsilon} =f_0\}$ and $D= \{\varphi_{0}=f_0\}$. Then $D_{\varepsilon}$ is increasing and $\cap_{\varepsilon>0} D_{\varepsilon}=D$ since $\varphi_0 \leq \varphi_{0,\varepsilon}\leq f_0$. If $x\in D$ then, for all small $s>0$, 
	\begin{flalign*}
		\dot{\varphi}_0(x)=\lim_{t\rightarrow 0} \frac{\varphi_t(x)-f_0(x)}{t} \leq \dot{\varphi}_{0,\varepsilon}(x) \leq  \frac{\varphi_{s,\varepsilon}(x)-\varphi_{0,\varepsilon}(x)}{s}, 
	\end{flalign*}
    where in the last inequality we use the convexity of the geodesic in $t$.
	Letting first $\varepsilon\to 0$ and then $s\to 0$ shows that $\dot{\varphi}_{0,\varepsilon}(x)$ converges to $\dot{\varphi}_0(x)$. If $x\notin D$ then $x\notin D_{\varepsilon}$, for $\varepsilon>0$ small enough. In this case the convergence we want to prove is trivial. 
\end{proof}

\begin{theorem}\label{thm: definition of dp}
Let $\varphi_0, \varphi_1\in \mathcal{H}_\theta$ and $\varphi_{i,\varepsilon}=P_{\omega_{\varepsilon}}(f_i)$, $i=0,1$. Let $d_{p, \varepsilon}$ be the Mabuchi distance w.r.t.  $\omega_\varepsilon$ defined in \eqref{distance form}. Then
$$
\lim_{\varepsilon \rightarrow 0} d^p_{p,\varepsilon}(\varphi_{0, \varepsilon}, \varphi_{1, \varepsilon}) = \int_X |\dot{\varphi}_0|^p (\theta+dd^c \varphi_0)^n=\int_X |\dot{\varphi}_1|^p (\theta+dd^c \varphi_1)^n,
$$ 
where $\varphi_t$ is the weak Mabuchi geodesic connecting $\varphi_0$ and $\varphi_1$. 
\end{theorem}
Compared to \cite{DiNezza_Guedj_2016_CM} our approach is slightly different. We also emphasize that by \cite[Example 4.5]{DiNezza_2015_JGA}, there are  functions in $\Ec^p(X,\theta)$ which are not in $\Ec^p(X,\omega)$.

\begin{proof}
Let $\varphi_{t,\varepsilon}$ denote the $\omega_\varepsilon$-geodesic joining $\varphi_{0,\varepsilon}$ and $\varphi_{1,\varepsilon}$.  Set $D_{\varepsilon}=\{\varphi_{0,\varepsilon} =f_0\}$ and $D= \{\varphi_{0}=f_0\}$. Combining  \eqref{eq: dp formula} and Theorem \ref{thm: regularity of psh envelope} we obtain
$$V_\varepsilon d_{p, \varepsilon}^p (\varphi_{0, \varepsilon}, \varphi_{ 1,\varepsilon}) = \int_X |\dot{\varphi}_{0, \varepsilon}|^p  (\omega_{\varepsilon}+dd^c \varphi_{0, \varepsilon})^n= \int_{D_\varepsilon} |\dot{\varphi}_{0, \varepsilon}|^p (\omega_\varepsilon+dd^c f_0)^n.$$
Since $|\varphi_{0, \varepsilon}-\varphi_{1,\varepsilon}|\leq \sup_X |f_0-f_1|$ and $f_0-f_1$ is bounded, \eqref{eq: Lip} ensures that $\dot{\varphi}_{0,\varepsilon}$ is uniformly bounded. It follows from Lemma \ref{lem: density converge} and Lemma \ref{lem: time derivative converge} that the functions $\id_{D_\varepsilon} |\dot{\varphi}_{0,\varepsilon}|^p \rho_\varepsilon$ and $ \id_{D} |\dot{\varphi}_{0}|^p \rho$ are uniformly bounded and $\id_{D_\varepsilon} |\dot{\varphi}_{0,\varepsilon}|^p \rho_\varepsilon$ converges pointwise to $ \id_{D} |\dot{\varphi}_{0}|^p \rho$. We also observe that $V_\varepsilon $ decreases to $\vol(\theta)=1$.
Lebesgue's dominated convergence theorem then  yields
$$
\lim_{\varepsilon\to 0}  d_{p, \varepsilon}^p (\varphi_{0, \varepsilon}, \varphi_{ 1,\varepsilon}) = \int_{D} |\dot{\varphi}_0|^p (\theta+dd^c f_0)^n=\int_{X} |\dot{\varphi}_0|^p (\theta+dd^c \varphi_0)^n,
$$ 
where in the last equality we use Theorem \ref{thm: regularity of psh envelope}. This shows the first equality in the statement. The second one is obtained by reversing the role of $\varphi_0$ and $\varphi_1$.  
\end{proof}
\begin{definition}\label{def: dp bounded Lap}
	Given $\varphi_0,\varphi_1 \in \mathcal{H}_{\theta}$, we define 
	$$
	d_p(\varphi_0,\varphi_1) := \lim_{\varepsilon\to 0} d_{p,\varepsilon}(\varphi_{0,\varepsilon},\varphi_{1,\varepsilon}).
	$$
	The limit exists and is independent of the choice of $\omega$ as shown in Theorem \ref{thm: definition of dp}.
\end{definition}
\begin{lemma}
$d_p$ is a distance on $\mathcal{H}_{\theta}$. 
 \end{lemma}

\begin{proof}
The triangle inequality immediately follows from the fact that $d_{p, \varepsilon}$ is a distance. From \cite[Theorem 5.5]{Darvas_2015_AIM} we know that $$d_{p, \varepsilon}^p (\varphi_{0, \varepsilon}, \varphi_{1, \varepsilon})\geq \frac{1}{C}\, I_{p,\omega_\varepsilon} (\varphi_{0, \varepsilon}, \varphi_{1, \varepsilon}), \qquad C>0.$$
Also, by Lemma \ref{lem: Ip convergence} we have $\lim_{\varepsilon \rightarrow 0}  I_{p,\omega_\varepsilon} (\varphi_{0, \varepsilon}, \varphi_{1, \varepsilon}) =  I_{p, \theta} (\varphi_{0}, \varphi_1) $. It  follows from the domination principe (see \cite{Darvas_DiNezza_Lu_2018_CM}, \cite{Bloom_Levenberg_2012_PA}) that 
$$ 
I_{p, \theta} (\varphi_{0}, \varphi_1)= 0 \Leftrightarrow \varphi_0=\varphi_1.
$$
 Hence, $d_p$ is non-degenerate.
\end{proof}

\subsection{Extension of $d_p$ on $\Ec^p(X,\theta)$}

  The following  comparison between $I_p$ and $d_p$ was established in \cite[Theorem 3]{Darvas_2015_AIM} in the K\"ahler case. 
\begin{prop}\label{prop: comparison dp and Ip}
Given $\varphi_0,\varphi_1 \in \mathcal H_\theta$ there exists a constant $C>0$ (depending only on $n$) such that
\begin{equation}\label{eq: comparision I_p and d_p}
\frac{1}{C}I_p(\varphi_0,\varphi_1) \leq d_p^p(\varphi_0,\varphi_1) \leq C I_p(\varphi_0,\varphi_1).
\end{equation}
\end{prop}

\begin{proof}
By Darvas \cite[Theorem 3]{Darvas_2015_AIM} we know that
$$\frac{1}{C}I_{p, \omega_\varepsilon}(\varphi_{0,\varepsilon}, \varphi_{1,\varepsilon}) \leq d_{p, \varepsilon}^p(\varphi_{0,\varepsilon}, \varphi_{1,\varepsilon}) \leq C I_{p,\omega_\varepsilon}(\varphi_{0,\varepsilon},\varphi_{1,\varepsilon}).$$
 Letting $\varepsilon$ to zero and using Lemma \ref{lem: Ip convergence} and Definition \ref{def: dp bounded Lap} we get  \eqref{eq: comparision I_p and d_p}. 
\end{proof}

\bigskip

Now, let $\varphi_0, \varphi_1\in \mathcal{E}^p(X, \theta)$. Let $\{f_{i,j}\}$ be a sequence of smooth functions decreasing to $\varphi_i$, $i=0,1$. We then clearly have that $\varphi_{i,j}:=P_\theta(f_{i,j})\in \mathcal{H}_\theta$ and $P_\theta(f_{i,j})\searrow \varphi_i$. 
\begin{lemma}
	\label{lem: limit Ep exists}
	The sequence $d_p(\varphi_{0,j},\varphi_{1,j})$  converges and the limit is independent of the choice of the approximants $f_{i,j}$. 
\end{lemma}
\begin{proof}
	Set $a_j:=  d_p (\varphi_{0,j}, \varphi_{1,j}) $. By the triangle inequality  and Proposition \ref{prop: comparison dp and Ip} we have
\begin{eqnarray*}
a_j & \leq &  d_p (\varphi_{0,j}, \varphi_{0,k})+  d_p (\varphi_{0,k}, \varphi_{1,k})+  d_p (\varphi_{1,k}, \varphi_{1,j}) \\
&\leq &  a_k +C\left( I_p^{1/p} (\varphi_{0,j}, \varphi_{0,k}) +  I_p^{1/p} (\varphi_{1,j}, \varphi_{1,k}) \right),
\end{eqnarray*}
where $C>0$ depends only on $n,p$. Hence $$|a_j-a_k| \leq C\left( I_p^{1/p} (\varphi_{0,j}, \varphi_{0,k}) +  I_p^{1/p} (\varphi_{1,j}, \varphi_{1,k}) \right).$$
By \cite[Theorem 1.6 and Proposition 1.9]{Guedj_Lu_Zeriahi_2017_JDG}, it then follows that $|a_j-a_k|\rightarrow 0$ as $j,k\rightarrow +\infty$. This proves that the sequence $d_p(\varphi_{0,j},\varphi_{1,j})$ is Cauchy, hence it converges.

Let $\tilde{\varphi}_{i,j}= P_\theta(\tilde{f}_{i,j})$  be another sequence in $\mathcal{H}_\theta$ decreasing to $\varphi_i$, $i=0,1$. Then applying the triangle inequality several times we get
$$ 
d_p (\varphi_{0,j}, \varphi_{1,j}) \leq  d_p (\varphi_{0,j}, \tilde{\varphi}_{0,j}) + d_p (\tilde{\varphi}_{0,j}, \tilde{\varphi}_{1,j}) + d_p (\tilde{\varphi}_{1,j}, \varphi_{1,j}), 
$$
and thus
$$| d_p (\varphi_{0,j}, \varphi_{1,j}) -d_p (\tilde{\varphi}_{0,j}, \tilde{\varphi}_{1,j})|   \leq C\left( I_p^{1/p} (\varphi_{0,j}, \tilde{\varphi}_{0,j}) +  I_p^{1/p} (\varphi_{1,j}, \tilde{\varphi}_{1,j}) \right).$$
It then follows again from \cite[Theorem 1.6 and Proposition 1.9]{Guedj_Lu_Zeriahi_2017_JDG} that the limit does not depend on the choice of the approximants.
\end{proof}

Given $\varphi_0, \varphi_1 \in \mathcal{E}^p(X, \theta)$, we then define
\begin{equation*}
d_p(	\varphi_0,\varphi_1) := \lim_{j\rightarrow +\infty}\, d_p (P_\theta(f_{0,j}), P_\theta(f_{1,j})).
\end{equation*}

\begin{prop}\label{prop: inherit} $d_p$ is a distance on $\mathcal{E}^p(X, \theta)$ and the inequalities comparing $d_p$ and $I_p$ on $\mathcal{H}_\theta$ \eqref{eq: comparision I_p and d_p} hold on $\mathcal{E}^p(X, \theta)$. Moreover, if $u_j\in \Ec^p(X,\theta)$ decreases to $u\in \Ec^p(X,\theta)$ then $d_p(u_j,u)\to 0$.  
\end{prop}

\begin{proof}
By definition of $d_p$ on $\mathcal{E}^p(X,\theta)$ we infer that the comparison between $d_p$ and $I_p$ in Proposition \ref{prop: comparison dp and Ip} holds on $\Ec^p(X,\theta)$. From this and the domination principle \cite{Darvas_DiNezza_Lu_2018_CM} we deduce that $d_p$ is non-degenerate. The last statement follows from \eqref{eq: comparision I_p and d_p} and \cite[Proposition 1.9]{Guedj_Lu_Zeriahi_2017_JDG}.
\end{proof}

The next result was proved in \cite[Lemma 3.4]{Berman_Darvas_Lu_2016_Minimizer} for the K\"ahler case. 
\begin{lemma}
	\label{lem: distance formula one smooth endpoint}
	Let $u_t$ be the Mabuchi geodesic joining $u_0\in \mathcal{H}_{\theta}$ and $u_1\in \Ec^p(X,\theta)$. Then 
	$$
	d_p^p(u_0,u_1) = \int_X |\dot{u}_0|^p (\theta+dd^c u_0)^n.
	$$
\end{lemma}

\begin{proof}
	We first assume that $u_0\geq u_1+1$. We approximate $u_1$  from above by $u_1^j\in \mathcal{H}_{\theta}$ such that $u_1^j\leq u_0$, for all $j$. Let $u_{t}^j$ be the Mabuchi geodesic joining $u_0$ to $u_1^j$.  Note that $u_t^j\geq u_t$ and that $\dot{u}_t^j\leq 0$. By Theorem \ref{thm: definition of dp},
$$ 
d_p^p(u_0,u_1^j) = \int_X (-\dot{u}_{0}^j)^p \theta_{u_0}^n.
$$
Also, $\dot{u}_{0}^j$ decreases to $\dot{u}_0$, hence the monotone convergence theorem and Proposition \ref{prop: inherit} give 
$$
d_p^p(u_0,u_1) = \int_X (- \dot{u}_0)^p \theta_{u_0}^n <+\infty.
$$
In particular $|\dot{u}_0|^p \in L^1(X, \theta_{u_0}^n)$. 

For the general case we can find  a constant $C>0$ such that $u_1\leq u_0+C$ since $u_0$ has minimal singularities. Let $w_t$ be the Mabuchi geodesic joining $u_0$ and $u_1-C-1$. Note that $w_t\leq u_t^j$ since $w_1=u_1-C-1< u_1\leq  u_1^j$ and $w_0=u_0=u_{0}^j$ and $\dot{w}_t\leq 0$. It then follows that 
$$
\dot{w}_0 \leq \dot{u}_0^j \leq u_1^j-u_0 \leq (u_1^j-V_{\theta}) + (V_{\theta}-u_0)\leq   \sup_X u_1^j  + \sup_{X} (V_{\theta}-u_0)\leq C_1,
$$
for a uniform constant $C_1>0$. In the second inequality above we use the fact that the Mabuchi geodesic $u_t^j$ connecting $u_0$ to $u_1^j$ is convex in $t$, while in the last inequality we use the fact that $u_0$ has minimal singularities. 

The previous inequalities then yield $|\dot{u}_0^j|^p \leq C_2 + 2^{p-1}|\dot{w}_0|^p$, where $C_2$ is a uniform constant. On the other hand by Theorem \ref{thm: definition of dp} we have
$$ 
d_p^p(u_0,u_1^j) = \int_X |\dot{u}_{0}^j|^p \theta_{u_0}^n.
$$
We claim that $|\dot{u}_0^j|^p$ converges a.e. to $|\dot{u}_0|^p$. Indeed, the convergence is pointwise at points $x$ such that $u_1(x)>-\infty$; but the set $\{u_1=-\infty\}$ has Lebesgue measure zero. Also, the above estimate ensures that $|\dot{u}_0^j|^p$ are uniformly bounded by $2^{p-1}(-\dot{w}_0)^p+C_2$ which is integrable w.r.t. the measure $\theta_{u_0}^n$ since $\int_X |\dot{w}_0|^p \theta_{u_0}^n= d_p^p(u_0,u_1-C-1)<+\infty$. Proposition \ref{prop: inherit} and Lebesgue's dominated convergence theorem then give the result.
\end{proof}

\begin{prop} \label{prop: diamond inequality and Pythagorean formula}
If $u,v\in\mathcal{E}^p(X, \theta)$ then
\begin{itemize}
	\item[(i)] $d_p^p (u,v) =d_p^p(u,P_{\theta}(u,v)) +d_p^p (v,P_{\theta}(u,v))$ and 
	\item[(ii)] $d_p(u, \max(u,v)) \geq d_p(v, P_{\theta}(u,v))$. 
\end{itemize}
\end{prop}
We recall that from \cite[Theorem 2.13]{Darvas_DiNezza_Lu_2018_CM} $P_\theta(u,v)\in \Ec^p(X, \theta)$.  
The identity in the first statement is known as the Pythagorean formula and it was established in the K\"ahler case by Darvas \cite{Darvas_2015_AIM}. The second statement was proved for $p=1$ in \cite{Darvas_Dinezza_Lu_2018_L1} using the differentiability of the Monge-Amp\`ere energy. As will be shown in Proposition \ref{prop: approximation vs MA} our definition of $d_1$ and the one in \cite{Darvas_Dinezza_Lu_2018_L1} do coincide.
\begin{proof}
To prove the Pythagorean formula we first assume that $u=P_{\theta}(f),v=P_{\theta}(g)\in \Htheta$. Set $u_{\varepsilon}:=P_{\omega_{\varepsilon}}(f)$, $v_{\varepsilon}:=P_{\omega_{\varepsilon}}(g)$. It follows from Lemma \ref{lem: Pythagorean Darvas} that 
\begin{eqnarray*}
d_{p,\varepsilon}^p (u_{\varepsilon}, v_{\varepsilon})& =&  d_{p,\varepsilon}^p (u_{\varepsilon}, P_{\omega_{\varepsilon}}(u_{\varepsilon},v_{\varepsilon})) + d_{p,\varepsilon}^p (v_{\varepsilon}, P_{\omega_{\varepsilon}}(u_{\varepsilon},v_{\varepsilon}))\\
&=& d_{p,\varepsilon}^p (u_{\varepsilon}, P_{\omega_{\varepsilon}}(\min(f,g)) + d_{p,\varepsilon}^p (v_{\varepsilon}, P_{\omega_{\varepsilon}}(\min(f,g)),
\end{eqnarray*}
where in the last identity we use that $P_{\omega_{\varepsilon}}(u_{\varepsilon},v_{\varepsilon})=P_{\omega_{\varepsilon}}(\min(f,g))$. It follows from Lemma \ref{lem: min} that $dd^c \min(f,g) \leq C\omega$. Applying Theorem \ref{thm: definition of dp} we obtain $(i)$ for this case. To treat the general case, let $u_j=P_{\theta}(f_j),v_j=P_{\theta}(g_j)$ be sequences in $\mathcal{H}_{\theta}$ decreasing to $u,v$. By Lemma \ref{lem: min}, $P_{\theta}(u_j,v_j)=P_{\theta}(\min(f_j,g_j))\in \Htheta$ and it decreases to $P_{\theta}(u,v)$. Then $(i)$ follows from the first step and Proposition \ref{prop: inherit} since 
$$|d_p(u_j, v_j)  - d_p(u,v)| \leq d_p(u_j, u) +d_p(v, v_j). $$

To prove the second statement, in view of Proposition \ref{prop: inherit}, we can assume that $u=P_{\theta}(f), v=P_{\theta}(g) \in \mathcal{H}_{\theta}$. By Lemma \ref{lem: distance formula one smooth endpoint} we have
$$
d_p^p(u,\max(u,v)) = \int_X |\dot{u}_0|^p \theta_u^n,
$$
where $t\mapsto u_t$ is the Mabuchi geodesic joining $u_0=u$ to $u_1=\max(u,v)$.

Let $\varphi_t$ be the Mabuchi geodesic joining $\varphi_0=P_{\theta}(u,v)$ to $\varphi_1=v$. We note that $0\leq \dot{\varphi}_0\leq v-P(u,v)$. Indeed $\dot{\varphi}_0\geq 0$ since $\varphi_0\leq \varphi_1$ while the second inequality follows from the convexity in $t$ of the geodesic. Using this observation and the fact that $\varphi_t\leq u_t$ we obtain
 $$
 \id_{\{P(u,v)=u\}} \dot{\varphi}_0 \leq \id_{\{P(u,v)=u\}} \dot{u}_0, \; \text{and}\;  \id_{\{P(u,v)=v\}} \dot{\varphi}_0 =0.
 $$
 Since $P_{\theta}(u,v)=P_{\theta}(\min(f,g))$ with $dd^c \min (f,g) \leq C \omega$, Theorem \ref{thm: regularity of psh envelope}, Theorem \ref{thm: definition of dp}, and \cite[Lemma 4.1]{Guedj_Lu_Zeriahi_2017_JDG} then yield
 \begin{flalign*}
 	d_p^p(P_{\theta}(u,v),v)& = \int_X \dot{\varphi}_0^p (\theta+ dd^c \varphi_0)^n \leq  \int_{\{P(u,v)=u\}} \dot{\varphi}_0^p (\theta+ dd^c u)^n\\
 	& \leq  \int_{\{P(u,v)=u\}} \dot{u}_0^p (\theta+ dd^c u)^n \leq d_p^p(u,\max(u,v)). 
 \end{flalign*}
\end{proof}

\begin{remark}
By Proposition \ref{prop: diamond inequality and Pythagorean formula} we have a ``Pythagorean inequality'' for $\max$: 
$$
d_p^p(u,\max(u,v)) + d_p^p(v,\max(u,v)) \geq d_p^p(u,v), \ \forall u,v\in \Ec^p(X,\theta). 
$$
\end{remark}

\subsection{Completeness of $(\mathcal{E}^p(X,\theta),d_p)$}
In the sequel we fix a smooth volume form $dV$ on $X$ such that $\int_X dV=1$.
\begin{lemma}\label{control sup}
Let $u\in \mathcal{E}^p(X, \theta)$ and let $\phi$ be a $\theta$-psh function with minimal singularities, $\sup_X \phi=0$ satisfying $\theta_\phi^n=dV$. Then there exist uniform constants $C_1=C_1(n,\theta)$ and  $C_2=C_2(n)>0$  such that
$$|\sup_X u |\leq C_1+C_2 d_p(u, \phi).$$

\end{lemma}

\begin{proof}
Using the H\"older inequality and  \cite[Proposition 2.7]{Guedj_Zeriahi_2005_JGA})  we obtain
\begin{eqnarray*}
|\sup_X u| &\leq &  \int_X |u-\sup_X u| dV +\int_X |u| dV \leq  A+ \left(  \int_X |u|^p dV \right)^{1/p}\\
&\leq &  A+ \left( \|u-\phi\|_{L^p(dV)}+\|\phi\|_{L^p(dV)}  \right).
\end{eqnarray*}
By Proposition  \ref{prop: inherit},
$$
\int_X |u-\phi|^p dV = \int_X |u-\phi|^p \theta_\phi^n \leq  I_p (u, \phi)\leq C(n) d_p^p (u, \phi).
$$
Combining the above inequalities we get the conclusion.
\end{proof}
\begin{theorem}\label{thm: main}
The space $(\mathcal{E}^p(X, \theta), d_p)$ is a complete geodesic metric space which is the completion of $(\mathcal{H}_{\theta},d_p)$. 
\end{theorem}
\begin{proof}
Let $(\varphi_j)\in \mathcal{E}^p(X, \theta)^{\mathbb{N}}$ be a Cauchy sequence for $d_p$. Extracting and relabelling we can assume that there exists a subsequence $(u_j)\subseteq (\varphi_j)$ such that
$$d_p(u_j, u_{j+1} )\leq 2^{-j}.$$
Define $v_{j,k}:= P_\theta (u_j, \dots, u_{j+k})$ and observe that it is decreasing in $k$. Also, by Proposition \ref{prop: diamond inequality and Pythagorean formula} $(i)$ and the triangle inequality,
$$d_p (u_j, v_{j,k})= d_p (u_j, P_\theta(u_j, v_{j+1,k}))\leq d_p (u_j,v_{j+1,k} ) \leq 2^{-j} + d_p (u_{j+1},v_{j+1,k} ).$$
Hence 
$$d_p (u_j, v_{j,k}) \leq \sum_{\ell=j }^{k-1} 2^{-\ell}\leq 2^{-j+1}.$$
In particular $I_p(u_j, v_{j,k})$ is uniformly bounded from above. We then infer that $v_{j,k}$ decreases to $ v_j\in \PSH(X,\theta)$ as $k\rightarrow +\infty$ and a combination of Proposition \ref{prop: inherit} and \cite[Proposition 1.9]{Guedj_Lu_Zeriahi_2017_JDG} gives
\begin{equation}\label{eq: completeness}
	d_p(u_j,v_j)\leq 2^{1-j}, \ \forall j. 
\end{equation}
Let $\phi$ be the unique $\theta$-psh function with minimal singularities such that $\sup_X \phi=0$ and $\theta_{\phi}^n=dV$. By Lemma \ref{control sup},
\begin{eqnarray*}
|\sup_X v_j| &\leq & C_1+C_2 d_p(v_j, \phi)\leq C_1+C_2 \left(d_p(v_j, u_1)+ d_p(u_1, \phi)\right)\\
&\leq &  C_1+C_2 \left(d_p(v_j, u_j)+ d_p(u_j, u_1)+ d_p(u_1, \phi)\right)\\
&\leq &C_1+ C_2\,(4+d_p(u_1, \phi)).
\end{eqnarray*}
It thus follows that $v_j$ increases a.e. to a $\theta$-psh function $v$. By the triangle inequality we have
$$
d_p(\varphi_j, v)\leq d_p(\varphi_j, u_j)+ d_p(u_j, v_{j})+ d_p(v_j, v).
$$
Since $(\varphi_j)$ is Cauchy,  $ d_p(\varphi_j, u_j)\to 0$. By   \cite[Proposition 1.9]{Guedj_Lu_Zeriahi_2017_JDG} and Proposition \ref{prop: inherit} we have $d_p(v_j,v)\to 0$. These facts together with \eqref{eq: completeness} yield $d_p(\varphi_j,v)\to 0$, hence $(\mathcal{E}^p(X, \theta), d_p)$ is a complete metric space.

Also, any $u\in \Ec^p(X,\theta)$ can be approximated from above by functions $u_j\in \mathcal{H}_{\theta}$ such that $d_p(u_j,u)\to 0$ (Proposition \ref{prop: inherit}). It thus follows that $(\Ec^p(X, \theta),d_p)$ is the metric completion of $\mathcal{H}_{\theta}$. 

Let now  $u_t$ be the Mabuchi geodesic joining $u_0,u_1\in \Ec^p(X,\theta)$. We are going to prove that, for all $t\in [0,1]$, 
\begin{equation*}
d_p(u_t,u_s) =|t-s|d_p(u_0,u_1).
\end{equation*}
We claim that for all $t\in [0,1]$,
\begin{equation}
	\label{eq: geodesic 1}
	d_p(u_0,u_t)= t d_p(u_0,u_1) \ \text{and}\  d_p(u_1,u_t)= (1-t) d_p(u_0,u_1).
\end{equation}
 We first  assume that $u_0,u_1\in \mathcal{H}_{\theta}$. The Mabuchi geodesic joining $u_0$ to $u_t$ is given by $w_{\ell}=  u_{t\ell}$, $\ell\in [0,1]$. Lemma \ref{lem: distance formula one smooth endpoint} thus gives
$$
d_p^p(u_0,u_t) = \int_X |\dot{w}_{0}|^p \theta_{u_0}^n = t^p \int_X |\dot{u}_0|^p\theta_{u_0}^n=t^p d_p^p(u_0,u_1),
$$
proving the first equality in \eqref{eq: geodesic 1}. The second one is proved similarly.   

We next prove the claim for $u_0,u_1\in \mathcal{E}^p(X,\theta)$.  Let $(u_i^j), i=0,1, j\in \mathbb{N}$, be decreasing sequences of functions in $\mathcal{H}_{\theta}$ such that $u_i^j \downarrow u_i$, $i=0,1$. Let $u_t^j$ be the Mabuchi geodesic joining $u_0^j$ and $u_1^j$. Then $u_t^j$ decreases to  $u_t$. By the triangle inequality we have $|d_p(u_0^j, u_t^j)-d_p(u_0, u_t)|\leq d_p(u_0^j, u_0)+d_p(u_t, u_t^j)$. The claim thus follows from Proposition \ref{prop: inherit} and the previous step.

Now, if $0<t<s<1$ then applying twice \eqref{eq: geodesic 1} we get 
$$
d_p(u_t,u_s)= \frac{s-t}{s} d_p(u_0,u_s) = (s-t) d_p(u_0,u_1). 
$$ 
 \end{proof}
 
 We end this section by proving that the distance $d_1$ defined by approximation (see Definition \ref{def: dp bounded Lap})  coincides with the one defined in \cite{Darvas_Dinezza_Lu_2018_L1} using the Monge-Amp\`ere energy.
 \begin{prop}\label{prop: approximation vs MA}
 	Assume $u_0,u_1\in \Ec^1(X,\theta)$. Then 
 	$$
 	d_1(u_0,u_1)=E(u_0)+E(u_1)-2E(P(u_0,u_1)).
 	$$
 \end{prop}
 Here the Monge-Amp\`ere energy $E$ is defined as
 $$
 E(u) := \frac{1}{n+1} \sum_{j=0}^n \int_X (u-V_{\theta}) \theta_u^{j} \wedge \theta_{V_{\theta}}^{n-j}. 
 $$
 \begin{proof}
 We first assume that $u_0,u_1 \in \mathcal{H}_{\theta}$ and $u_0\leq u_1$. Let $[0,1] \ni t\mapsto u_t$ be the Mabuchi geodesic joining $u_0$ and $u_1$. By \cite[Theorem 3.12]{Darvas_DiNezza_Lu_2018_CM}, $t\mapsto E(u_t)$ is affine, hence for all $t\in [0,1]$,
 $$
\frac{ E(u_t)-E(u_0)}{t} = E(u_1)-E(u_0)  = \frac{E(u_1)-E(u_t)}{1-t}.
 $$
 Since $E$ is concave along affine curves (see \cite{Berman_Boucksom_Guedj_Zeriahi_2013_IHES}, \cite{Boucksom_Eyssidieux_Guedj_Zeriahi_2010_AM}, \cite[Theorem 2.1]{Darvas_Dinezza_Lu_2018_L1}) we thus have 
 $$
 \int_X \frac{u_t-u_0}{t} \theta_{u_0}^n \geq   E(u_1)-E(u_0) \geq \int_X \frac{u_1-u_t}{1-t} \theta_{u_1}^n.
 $$
 Letting $t\to 0$ in the first inequality and $t\to 1$ in the second one we obtain
 $$
 \int_X \dot{u}_0 \theta_{u_0}^n \geq E(u_1)-E(u_0) \geq \int_X \dot{u}_1 \theta_{u_1}^n. 
 $$
 By Theorem \ref{thm: definition of dp} we then have $$ d_1(u_0-u_1)= \int_X \dot{u}_0 \theta_{u_0}^n= \int_X \dot{u}_1 \theta_{u_1}^n= E(u_1)-E(u_0).$$

We next assume that $u_0,u_1\in \mathcal{H}_{\theta}$ but we remove the assumption that $u_0\leq u_1$. By Lemma \ref{lem: min}, $P(u_0,u_1) \in \mathcal{H}_{\theta}$.  By the Pythagorean formula (Proposition \ref{prop: diamond inequality and Pythagorean formula}) and the first step we have 
\begin{flalign*}
d_1(u_0,u_1) &=d_1(u_0,P(u_0,u_1)) + d_1(u_1,P(u_0,u_1)) \\
	& = E(u_0)-E(P(u_0,u_1)) + E(u_1)-E(P(u_0,u_1)).
\end{flalign*}

We now treat the general case. Let $(u_i^j), i=0,1, j\in \mathbb{N}$ be decreasing sequences of functions in $\mathcal{H}_{\theta}$ such that $u_i^j \downarrow u_i$, $i=0,1$. Then $P(u_0^j,u_1^j) \downarrow P(u_0,u_1)$. By \cite[Proposition 2.4]{Darvas_Dinezza_Lu_2018_L1}, $E(u_i^j) \to E(u_i)$, for $i=0,1$ and $E(P(u_0^j,u_1^j)) \to E(P(u_0,u_1))$ as $j\rightarrow +\infty$. The result thus follows from Proposition \ref{prop: inherit}, the triangle inequality and the previous step.
 \end{proof}



\begin{thebibliography}{99}

\bibitem{Bedford_Taylor_1976_IM}
E.~Bedford and B.~A. Taylor, {\em The {D}irichlet problem for a complex
  {M}onge-{A}mp\`ere equation}, Invent. Math., 37 (1976), pp.~1--44.

\bibitem{Bedford_Taylor_1982_AM}
E.~Bedford and B.~A. Taylor, {\em A new capacity for plurisubharmonic
  functions}, Acta Math., 149 (1982), pp.~1--40.

\bibitem{Berman_2018_MZ}
R.~J. Berman, {\em {From Monge--Amp{\`e}re equations to envelopes and geodesic
  rays in the zero temperature limit}}, Mathematische Zeitschrift,  (2018).

\bibitem{Berman_Boucksom_Guedj_Zeriahi_2013_IHES}
R.~J. Berman, S.~Boucksom, V.~Guedj, and A.~Zeriahi, {\em A variational
  approach to complex {M}onge-{A}mp\`ere equations}, Publ. Math. Inst. Hautes
  \'Etudes Sci., 117 (2013), pp.~179--245.

\bibitem{Berman_Darvas_Lu_2016_Minimizer}
R.~J. Berman, T.~Darvas, and C.~H. Lu, {\em {Regularity of weak minimizers of
  the K-energy and applications to properness and K-stability}},
  arXiv:1602.03114, accepted in Annales scientifiques de l'ENS,  (2018).

\bibitem{Berndtsson_2015_IM}
B.~Berndtsson, {\em A {B}runn-{M}inkowski type inequality for {F}ano manifolds
  and some uniqueness theorems in {K}\"ahler geometry}, Invent. Math., 200
  (2015), pp.~149--200.

\bibitem{Blocki_2012_bookIP}
Z.~B{\l}ocki, {\em On geodesics in the space of {K}\"ahler metrics}, in
  Advances in geometric analysis, vol.~21 of Adv. Lect. Math. (ALM), Int.
  Press, Somerville, MA, 2012, pp.~3--19.

\bibitem{Blocki_Kolodziej_2007_PAMS}
Z.~B{\l}ocki and S.~Ko{\l}odziej, {\em On regularization of plurisubharmonic
  functions on manifolds}, Proc. Amer. Math. Soc., 135 (2007), pp.~2089--2093.

\bibitem{Bloom_Levenberg_2012_PA}
T.~Bloom and N.~Levenberg, {\em Pluripotential energy}, Potential Anal., 36
  (2012), pp.~155--176.

\bibitem{Boucksom_2004_ASENS}
S.~Boucksom, {\em Divisorial {Z}ariski decompositions on compact complex
  manifolds}, Ann. Sci. \'Ecole Norm. Sup. (4), 37 (2004), pp.~45--76.

\bibitem{Boucksom_Eyssidieux_Guedj_Zeriahi_2010_AM}
S.~Boucksom, P.~Eyssidieux, V.~Guedj, and A.~Zeriahi, {\em Monge-{A}mp\`ere
  equations in big cohomology classes}, Acta Math., 205 (2010), pp.~199--262.

\bibitem{Calabi_1982_Seminar}
E.~Calabi, {\em Extremal {K}\"ahler metrics}, in Seminar on {D}ifferential
  {G}eometry, vol.~102 of Ann. of Math. Stud., Princeton Univ. Press,
  Princeton, N.J., 1982, pp.~259--290.

\bibitem{Chen_2000_JDG}
X.~Chen, {\em The space of {K}\"ahler metrics}, J. Differential Geom., 56
  (2000), pp.~189--234.

\bibitem{Chen_Cheng_2017_csckestimates}
X.~Chen and J.~Cheng, {\em {On the constant scalar curvature K{\"a}hler
  metrics, apriori estimates}}, arXiv:1712.06697,  (Preprint 2017).

\bibitem{Chen_Cheng_2017_csckexistence}
X.~Chen and J.~Cheng, {\em {On the constant scalar curvature K\"ahler metrics,
  existence results}}, arXiv:1801.00656,  (Preprint 2018).

\bibitem{Chen_Cheng_2017_csckgeneral}
X.~Chen and J.~Cheng, {\em {On the constant scalar curvature K\"ahler metrics,
  general automorphism group}}, {arXiv:1801.05907},  (Preprint 2018).

\bibitem{Chu_Tosatti_Weinkove_2017_AoPDE}
J.~Chu, V.~Tosatti, and B.~Weinkove, {\em On the {$C^{1,1}$} regularity of
  geodesics in the space of {K}\"ahler metrics}, Ann. PDE, 3 (2017), pp.~Art.
  15, 12.

\bibitem{Darvas_2015_AIM}
T.~Darvas, {\em The {M}abuchi geometry of finite energy classes}, Adv. Math.,
  285 (2015), pp.~182--219.

\bibitem{Darvas_2017_AJM}
T.~Darvas, {\em The {M}abuchi completion of the space of {K}\"ahler
  potentials}, Amer. J. Math., 139 (2017), pp.~1275--1313.

\bibitem{Darvas_2017_IMRN}
T.~Darvas, {\em Metric geometry of normal {K}\"ahler spaces, energy properness,
  and existence of canonical metrics}, Int. Math. Res. Not. IMRN,  (2017),
  pp.~6752--6777.

\bibitem{Darvas_DiNezza_Lu_2018_CM}
T.~Darvas, E.~Di~Nezza, and C.~H. Lu, {\em On the singularity type of full mass
  currents in big cohomology classes}, Compos. Math., 154 (02/2018),
  pp.~380--409.

\bibitem{Darvas_Dinezza_Lu_2018_APDE}
T.~Darvas, E.~Di~Nezza, and C.~H. Lu, {\em Monotonicity of nonpluripolar
  products and complex {M}onge--{A}mp\`ere equations with prescribed
  singularity}, Anal. PDE, 11 (06/2018), pp.~2049--2087.

\bibitem{Darvas_Dinezza_Lu_2018_L1}
T.~Darvas, E.~Di~Nezza, and C.~H. Lu, {\em {$L^1$ metric geometry of big
  cohomology classes}}, arXiv:1802.00087,  (Preprint 01/2018).

\bibitem{Darvas_Dinezza_Lu_2018_Logconcave}
T.~Darvas, E.~Di~Nezza, and C.~H. Lu, {\em {Log-concavity of volume and complex
  Monge-Amp\`ere equations with prescribed singularity}}, arXiv:072018,
  (Preprint 07/2018).

\bibitem{Darvas_Rubinstein_2016_JMSJ}
T.~Darvas and Y.~A. Rubinstein, {\em Kiselman's principle, the {D}irichlet
  problem for the {M}onge-{A}mp\`ere equation, and rooftop obstacle problems},
  J. Math. Soc. Japan, 68 (2016), pp.~773--796.

\bibitem{Darvas_Rubinstein_2017_JAMS}
T.~Darvas and Y.~A. Rubinstein, {\em Tian's properness conjectures and
  {F}insler geometry of the space of {K}\"ahler metrics}, J. Amer. Math. Soc.,
  30 (2017), pp.~347--387.

\bibitem{Demailly_1992_JAG}
J.-P. Demailly, {\em Regularization of closed positive currents and
  intersection theory}, J. Algebraic Geom., 1 (1992), pp.~361--409.

\bibitem{DiNezza_2015_JGA}
E.~Di~Nezza, {\em Stability of {M}onge-{A}mp\`ere energy classes}, J. Geom.
  Anal., 25 (2015), pp.~2565--2589.

\bibitem{DiNezza_Guedj_2016_CM}
E.~Di~Nezza and V.~Guedj, {\em Geometry and topology of the space of {K}\"ahler
  metrics on singular varieties}, Compos. Math., 154 (2018), pp.~1593--1632.

\bibitem{Donaldson_1999_bookAMS}
S.~K. Donaldson, {\em Symmetric spaces, {K}\"ahler geometry and {H}amiltonian
  dynamics}, in Northern {C}alifornia {S}ymplectic {G}eometry {S}eminar,
  vol.~196 of Amer. Math. Soc. Transl. Ser. 2, Amer. Math. Soc., Providence,
  RI, 1999, pp.~13--33.

\bibitem{Guedj_Lu_Zeriahi_2017_JDG}
V.~Guedj, C.~H. Lu, and A.~Zeriahi, {\em Plurisubharmonic envelopes and
  supersolutions}, arXiv:1703.05254, J. Differential Geom.,  (2017).

\bibitem{Guedj_Zeriahi_2005_JGA}
V.~Guedj and A.~Zeriahi, {\em Intrinsic capacities on compact {K}\"ahler
  manifolds}, J. Geom. Anal., 15 (2005), pp.~607--639.

\bibitem{Guedj_Zeriahi_2007_JFA}
V.~Guedj and A.~Zeriahi, {\em The weighted {M}onge-{A}mp\`ere energy of
  quasiplurisubharmonic functions}, J. Funct. Anal., 250 (2007), pp.~442--482.

\bibitem{Mabuchi_1987_OJM}
T.~Mabuchi, {\em Some symplectic geometry on compact {K}\"ahler manifolds.
  {I}}, Osaka J. Math., 24 (1987), pp.~227--252.

\bibitem{Semmes_1992_AJM}
S.~Semmes, {\em Complex {M}onge-{A}mp\`ere and symplectic manifolds}, Amer. J.
  Math., 114 (1992), pp.~495--550.

\bibitem{Szekelyhidi_2014_book}
G.~Sz{\'e}kelyhidi, {\em An introduction to extremal {K}\"ahler metrics},
  vol.~152 of Graduate Studies in Mathematics, American Mathematical Society,
  Providence, RI, 2014.

\bibitem{Yau_1978_CPAM}
S.~T. Yau, {\em On the {R}icci curvature of a compact {K}\"ahler manifold and
  the complex {M}onge-{A}mp\`ere equation. {I}}, Comm. Pure Appl. Math., 31
  (1978), pp.~339--411.

\end{thebibliography}

\end{document}